\documentclass[11pt,reqno,a4paper]{amsart}
\usepackage{paralist}

\usepackage{amssymb}
\usepackage{enumerate}
\usepackage[all]{xy}

\usepackage{pgf,tikz}
\usetikzlibrary{arrows}

\renewcommand{\ge}{\geqslant}
\renewcommand{\le}{\leqslant}
\renewcommand{\geq}{\geqslant}

\newcommand{\rto}{\dasharrow}

\newcommand{\PP}{\mathbb{P}}
\newcommand{\C}{\mathbb{C}}

\newcommand{\cL}{\mathcal{L}}
\newcommand{\cO}{\mathcal{O}}

\newcommand{\infnear}[1][]{>^{#1}}

\DeclareMathOperator{\ad}{ad}
\DeclareMathOperator{\kod}{kod}
\DeclareMathOperator{\Cr}{Cr}

\newtheorem{theorem}{Theorem}[section]
\newtheorem{lemma}[theorem]{Lemma}
\newtheorem{prop}[theorem]{Proposition}

\newtheorem*{problem*}{Problem}

\theoremstyle{definition}

\newtheorem{remark}[theorem]{Remark}
\newtheorem{example}[theorem]{Example}

\title[Cremona contractibility of unions of lines in the plane]{On Cremona contractibility of \\ unions of  lines in the plane}

\author{Alberto Calabri \and Ciro Ciliberto}

\email{alberto.calabri@unife.it}
\curraddr{Dipartimento di Matematica e Informatica,
Universit\`a degli Studi di Ferrara,
Via Machiavelli 30, 44121 Ferrara, Italy,
phone: +39-0532-974067, fax: +39-0532-974003}

\email{cilibert@mat.uniroma2.it}
\curraddr{Dipartimento di Matematica, Universit\`a degli Studi di Roma ``Tor Vergata'',
Via della Ricerca Scientifica, 00133 Roma, Italy, phone: +39-06-7259-4684,
fax: +39-06-7259-4699}

\thanks{2010 Mathematics Subject Classification: 14H50 (Primary), 14E07, 14N20 (Secondary).}

\keywords{Union of lines, contractible plane curves, log-Kodaira dimension}

\begin{document}

\begin{abstract}
We discuss the concept of Cremona contractible plane curves, with an historical account on the development of this subject. Then we classify Cremona contractible unions of $d\geqslant 12$ lines in the plane.
\end{abstract}

\maketitle

\section{Introduction}

The \emph{Cremona geometry} of the complex projective space $\PP^r$ consists in studying properties of subvarieties of $\PP^ r$ which are invariant under the action of the \emph{Cremona group} $\Cr_r$, i.e., the group of all  birational maps $\PP^r\rto\PP^r$.

Since $\Cr_1\cong {\rm PGL}(2,\C)$, the case $r=1$ reduces to the (non--trivial, but well known and widely studied) theory of invariants of finite sets of points of $\PP^ 1$ under the action of the projective linear group.  The case $r\geqslant 3$ has been very little explored, due to the fact that, among other things, very little is known about the structure of $\Cr_r$. Indeed, in this case we do not even know a reasonable set of generators of $\Cr_r$.  
The intermediate case $r=2$ is more accessible, and in fact it has been an object of study in the course of the last 150 years. The reason is that,  in this case, we have a good amount of 
information about $\Cr_2$. The first one is a famous result by Noether and Castelnuovo to the effect that $\Cr_2$ is generated by ${\rm PGL}(3,\C)$ and the \emph{standard quadratic map} 
\[
\sigma\colon [x,y,z] \in \PP^2\rto  [yz,zx,xy]\in \PP^2.
\]

A classical object of study, from this viewpoint, has been the classification of curves (or, more generally, of linear systems of curves) in $\PP^2$ up to the action of $\Cr_2$. If $\cL$ is a linear system of curves, its dimension is a \emph{Cremona invariant}, i.e., it is the same for all linear systems in the \emph{Cremona orbit} of $\cL$, i.e., the orbit of $\cL$ under the $\Cr_2$--action. 

The degree $d$ of the curves in $\cL$ instead (called the \emph{degree} of $\cL$ and denoted by $\deg(\cL)$), is not a Cremona invariant: for instance, if one applies to $\cL$ a general quadratic transformation (i.e., the composition of $\sigma$ with a general element of  ${\rm PGL}(3,\C)$), the degree of the transformed linear system is $2d$. However, one can define an important \emph{Cremona invariant} related to the degree, i.e., the \emph{Cremona degree} of $\cL$: this is the minimal degree of a linear system in the Cremona orbit of $\cL$. A (not necessarily unique, up to  projective transformations) linear system with minimal Cremona degree is called a \emph{Cremona minimal model}.

If $\cL$ has dimension  0, i.e., it consists of a unique curve $C$, the Cremona degree could be 0: this is the case if $C$ can be contracted to a set of points by a Cremona transformation. In this case one says that $C$ is \emph {Cremona contractible} or simply \emph{Cr--contractible}. If $C$ is Cr--contractible and reducible, it could be contracted to a set of distinct points. However it is easy to see that any finite set of points in $\PP^2$ can be mapped to a single point via a Cremona transformation. Thus, $C$ is Cr--contractible if and only if there is a Cremona transformation which contracts $C$ to a point of the plane.  If $\dim(\cL)\geqslant 1$, then the {Cremona degree} of $\cL$ is positive. 

The Cremona classification of {Cremona minimal models} of linear systems is a very classical subject. 
For example, it is a result which goes back to Noether (though with an incomplete proof) that a pencil of irreducible, rational plane curves is {Cremona equivalent} to the pencil of lines through a fixed point, i.e., pencils of rational plane curves have {Cremona degree} 1.  Similar results for linear systems of positive dimension of curves with positive genus have been classically proved, as we will see in \S \ref {S:history} which is devoted to an historical account on the subject. 

The general problem of classifying {Cremona minimal models} of \emph{irreducible} plane curves or linear systems (a linear system is said to be \emph{irreducible} if so is its general curve) has been open for more than one century, with several interesting contributions by various authors, among them it is worth mentioning  Giuseppe Marletta \cite{Marletta, Marletta2}, who pointed out important properties of adjoint linear systems to such models (see Theorem \ref {thm:marl}; for the definition of adjoint linear systems, see \S \ref {S:notation}). This problem however has been solved only recently in our paper \cite{CC}. 

The first step in this classification can be considered the characterization of Cr-contractible irreducible plane curves.  According to Enriques and Chisini in \cite[vol.~III, \S21, pp.~191--192]{EC}, the first result on this subject appeared in the paper \cite{CE} by Castelnuovo and Enriques in 1900:

\begin{theorem}[Castelnuovo--Enriques] \label{thm:CE}
An irreducible curve $C$ is Cr-contractible if and only if all adjoint linear systems to $C$ vanish.
\end{theorem}

Actually, Castelnuovo and Enriques in \cite{CE} claimed that the irreducibility assumption on $C$ can be relaxed to  $C$ being reduced, but Example 2 in \cite{CC2}, namely a general union of $d\ge9$ distinct lines with a point of multiplicity $d-3$, shows that this is not true.

Theorem \ref{thm:CE} is nowadays known as Coolidge's Theorem, because it appeared also in Coolidge's book \cite[p.~398]{Coolidge}, but the proof therein is not complete (see \S \ref{S:history}).
Theorem \ref{thm:CE} has been improved by Kumar and Murthy in 1982, cf.\ \cite{KumarMurthy}:

\begin{theorem}[Kumar--Murthy] \label{thm:KumarMurthy}
An irreducible plane curve $C$ is Cr-contractible if and only if the first two adjoint linear systems to $C$ vanish.
\end{theorem}

Using the modern language of \emph{pairs} of a curve on a smooth surface, 
one considers the pair $(S,\tilde C)$ where $S\to\PP^2$ is a birational morphism which resolves the singularities of $C$ and $\tilde C$ is the strict transform of $C$ on $S$.

Theorem \ref{thm:KumarMurthy} implies that the pair $(S,\tilde C)$ has log Kodaira dimension $\kod(S,\tilde C)=-\infty$ if and only if its second log plurigenus $P_2(S,\tilde C)$ vanishes (for the definitions, see again \S \ref {S:notation}). This can be seen as a log-analogue of Castelnuovo's rationality criterion for regular surfaces. Thus, for an irreducible plane curve $C$, the following four conditions are equivalent:
\begin{enumerate}[$(a)$]

\item  $C$ is Cr-contractible,

\item  $\kod(S,\tilde C)=-\infty$, 

\item  all adjoint linear systems to $C$ vanish,

\item the first two adjoint linear systems to $C$ vanish.

\end{enumerate}

Condition $(d)$ can be replaced by

\begin{enumerate}[$(d')$]
\item $P_2(S,\tilde C)=0$.
\end{enumerate}

The implications $(a)\Rightarrow(b)\Rightarrow(c)\Rightarrow(d)$ are either trivial or easy, and are true even for reducible and reduced plane curves (see \S \ref{S:notation}), while $(d)\Rightarrow(a)$ follows from Theorem \ref{thm:KumarMurthy}.

As for extensions of Kumar and Murthy's Theorem to reducible curves, the only known result so far is due to Iitaka \cite {Iit}:

\begin{theorem}[Iitaka] \label{thm:Iitaka}
Let $C$ be a reduced plane curve with two irreducible components.
Then, $C$ is Cr-contractible if and only if the first two adjoint linear systems to $C$ vanish.
\end{theorem}

By contrast, in \cite{CC2} we noted that $(a)$, $(b)$, $(c)$ and $(d)$ above are not equivalent for reducible, reduced plane curves.
As we said,  Example 2 in \cite{CC2} shows that $(b)$ and $(c)$ are not equivalent for reducible curves.
Furthermore, an example of Pompilj \cite {Pompilj} shows that  $(c)$ and $(d)$ are not equivalent for curves with three irreducible components, cf.\  \cite [Example 1] {CC2}. The same example shows that $(a)$ and $(d)$ are not equivalent for curves with three irreducible components.
Note, moreover, that Pompilj's example is the union of three Cr-contractible irreducible curves which turns out to be non-Cr-contractible and it shows the difficulty of proving the Cr-contractibiilty of reducible curves by proceeding by induction on the number of irreducible components of the curve, as one may be tempted to do. See the historical account in \S\ref{S:history}  for other difficulties encountered by several mathematicians in tackling this problem.

Concerning reducible curves,
the following theorem should also be recalled.

\begin{theorem}[Kojima--Takahashi, \cite {KT}]
Let $(S,D)$ be a pair where $S$ is a smooth rational surface and $D$ is a reduced curve on $S$ with at most four irreducible components. Then, $\kod(S,D)=-\infty$ if and only if $P_6(S,D)=0$.

Furthermore, if $(V,D')$ is the almost minimal model of $(S,D)$ in the sense of \cite[Definition 2.3]{KT}, and if the support of  $D'$ is connected, then $\kod(S,D)=-\infty$ if and only if $P_{12}(S,D)=0$.
\end{theorem}

However, Kojima e Takahashi do not relate  $\kod(S,D)=-\infty$ with contractibility of $D$.

In \cite{CC2} we posed the following:

\begin{problem*}
Is it true that a reduced plane curve $C$ is Cr-contractible if and only if $\kod(S,\tilde C)=-\infty$?
\end{problem*}

In this paper we address this problem when $C$ is a reduced union of lines, the first meaningful case, which,  we think, presents aspects of a general interest.
Our main result is the following (cf.\ Theorem \ref{thm:d>=12} for a more precise statement):

\begin{theorem}
Let $C$ be the union of $d\geqslant 12$ distinct lines. Then, all adjoint linear systems to $C$ vanish if and only if $C$ has a point of multiplicity $m\geq d-3$. Moreover, $\kod(S,\tilde C)=-\infty$ if and only if $C$ has a point of multiplicity $m\geq d-2$.
Finally, $C$ is Cr-contractible if and only if $\kod(S,\tilde C)=-\infty$.
\end{theorem}

\emph{A posteriori}, one has that $P_3(S,\tilde C)=0$ implies $\kod(S,\tilde C)=-\infty$ for $C$ a union of $d\geqslant12$ distinct lines with vanishing adjoints.
Moreover, it turns out that, for a union of $d\geq12$ distinct lines,  $(d)$ implies $(c)$. Note that this is not true if $d<12$: for example for the dual configurations of the flexes of a smooth cubic plane curve, which has degree 9, 12 triple points and no other singularity, the first two adjoint linear system vanish, but the third adjoint is trivial, hence non--empty.

The case of a reduced union of $d\leqslant11$ lines is also interesting but the classification of all cases with vanishing adjoints or with Kodaira dimension $-\infty$ is much more complicated, since it requires the analysis of many dozens of configurations.
We performed it  for $d\leqslant8$ and $d=11$,
the remaining cases are work in progress. So far all cases we found with Kodaira dimension $-\infty$ are also Cr-contractible.
We will not present here this long and tedious classification, but we intend to do it in a forthcoming paper.

This paper is organized as follows.
After the historical \S \ref {S:history}, we fix notation and definitions in \S \ref {S:notation}.
In \S \ref {s:adj}, we classify the union of $d\geq12$ distinct lines with vanishing adjoints.
Among them, we determine those with Kodaira dimension $-\infty$ (the latter set is strictly contained in the former) and we show that these are exactly the Cr-contractible ones.

\section{An historical account} \label{S:history}

The history of Theorem \ref{thm:CE} is surprisingly intricate and it is intertwined with the one of Noether--Castelnuovo's Theorem and with the problem of finding Cremona minimal models of plane curves and linear systems.

A short account of some proofs of  Noether--Castelnuovo Theorem can be found in \cite[pp.~227--228]{Alberich}, see also the historical remarks in Chapter 8 of the same book.
For more details on the classical literature, see \cite[pp.\ 390--391]{Hu} and \cite[vol.~III, \S20, pp.~175--177]{EC}.

The study of Cremona transformations $\gamma$ of $\PP^ 2$ is equivalent to the one of homaloidal nets: such a net is the image of the linear system of lines via $\gamma$. The degree of the homaloidal net $\cL$ associated to $\gamma\in \Cr_2$ is called the \emph{degree} of $\gamma$. Cremona transformations of degree $1$ are projective transformations, those of degree 2, the \emph{quadratic transformations}, correspond to homaloidal nets of conics, etc. 

If $\cL$ is an irreducible linear system of plane curves of degree $d$, with base points $P_0,\ldots, P_r$ of multiplicity at least $m_0,\ldots, m_r$, we may assume that $m_0\geqslant \cdots \geqslant m_r$. We will use the notation $(d;m_0,\ldots, m_r)$ to denote $\cL$. We may use exponential notations to denote repeated multiplicities.  For instance the homaloidal nets of conics are of the form $(2; 1^3)$, and the related quadratic transformation is said to be \emph{based} at the three simple base points of this net. 

The so called Noether--Castelnuovo's Theorem was apparently first  stated by Clifford in the fragment  \cite {Cl} of 1869. However Clifford gave no real proof of it, rather he  presented a plausibility argument based on the analysis of Cremona transformations of degree $d\leqslant 8$. 
Immediately after, in 1870,  Noether \cite {Noe} and Rosanes \cite {Ros} independently came up with a more promising approach. They correctly observed that for a homaloidal net $\cL=(d;m_0,\ldots, m_r)$ of degree $d>1$ one has $m_0+m_1+m_2>d$ (this is now called \emph{Noether's inequality}). Then, they observed, if one performs a quadratic transformation based at $P_0,P_1,P_2$, the homaloidal net $\cL$ is transformed in another of degree $d'=2d-(m_0+m_1+m_2)<d$. By repeating this argument, the degree of $\cL$ can be dropped to 1, proving the theorem. 

The problem with this argument is the existence of an \emph{irreducible} net of conics through $P_0,P_1,P_2$. This is certainly the case if $P_0,P_1,P_2$ are distinct, since then they cannot be collinear by 
Noether's inequality. The same argument applies also if $P_1$ is infinitely near to $P_0$ and $P_2$ is distinct, but problems may arise if both $P_1,P_2$ are infinitely near to $P_0$. The first difficulty appears if $P_1,P_2$ are infinitely near to $P_0$ in \emph{different directions}. This was noted by Noether himself, 
who filled up this gap in the paper \cite {Noe2}. After this the proof was considered to be correct and the theorem well established. Afterwords a considerable series of papers appeared, by several authors, like Bertini, Castelnuovo, Guccia, G.~Jung, Martinetti, Del Pezzo, De Franchis, C.~Segre (in chronological order), and others. Based on Noether's argument, they pursued the classification of Cremona minimal models of irreducible linear systems of positive dimension of curves of low genus. 

It was only  in 1901 that C.\ Segre pointed out a more subtle gap in Noether's argument, when:\\
\begin{inparaenum}[$\bullet$]
\item $P_2$ is infinitely near to $P_1$, which in turn is infinitely near to $P_0$, and\\
\item $P_2$ is \emph{satellite} to $P_0$, i.e.\ $P_2$ is proximate also to $P_0$.\\
\end{inparaenum}
In other words, $P_2$ is infinitely near to $P_0$ along a \emph{cuspidal} branch. Segre's criticism seemed to be a very serious one, since he presented a series of homaloidal nets of increasing degrees, whose degree cannot be lower by using quadratic transformations.  

According to Coolidge in \cite[p.~447]{Coolidge}, ``it is said that Noether shed tears when he heard of this'',
but, as Coolidge goes on, ``there was no need to do so''.
Indeed, promptly after Segre's criticism, in the same year 1901, Castelnuovo showed in \cite{Castelnuovo} how to decompose a non-linear Cremona transformations as a composition of  \emph{de Jonqui\`eres maps} (related to homaloidal nets of the type $(d; d-1, 1^ {2(d-1)})$, and the de Jonqui\`eres map will be said to be \emph{centered} at the base points of the net), which, in turn, decompose in products of quadratic ones, as showed by Segre in a footnote to \cite{Castelnuovo}. Just one year later, Ferretti, a student of Castelnuovo's, filled up in \cite{Ferretti} the gap also in aforementioned  papers about the classification of linear systems of low genus.
It turned out that, even if the proofs were incomplete, all statements were correct.

Castelnuovo's proof really contains a new idea: it is based on the remark that, if $\cL$ is a positive dimensional linear system of rational plane curves, then all adjoint linear systems to $\cL$ vanish. It is this property that ultimately implies that $\cL$ has base points of large enough multiplicity so that the degree of  $\cL$ can be decreased by means of de Jonqui\`eres transformations.
This idea, according to Castelnuovo himself, came from the joint work \cite{CE} with Enriques of the year before, concerning rational and ruled \emph{double planes}, i.e., rational and ruled  double coverings of $\PP^2$.
Indeed, Castelnuovo ed Enriques stated  in \cite{CE}  that a double plane is rational or ruled if and only if all the adjoint linear systems of index $i\ge2$ to the branch curve of (the canonical desingularization of) a double plane vanish (see \cite {Cal1, Cal} for a more precise statement). 

In the last page of \cite{CE}, Castelnuovo ed Enriques stated Theorem \ref{thm:CE} (with the wrong assumption that $C$ can be reducible): they do not really give a proof, they simply claim that it consists in computations similar to others done in that paper.
In addition, they remarked that the same technique could be useful in the classification of linear systems of plane curves with low genus, as Castelnuovo and Ferretti effectively did.

Note however that Segre's criticism applied also to the classification in \cite{CE}, as Castelnuovo admitted in the first page of \cite{Castelnuovo}.
Even if Castelnuovo suggested in \cite{Castelnuovo} that the gap could be fixed by arguments similar to those in \cite{Castelnuovo}, it seems that nobody did that, until Conforto in 1938 in \cite{Conforto}, cf.\ \cite[p.~458]{EConforto}.
It turned out only recently that Castelnuovo--Enriques--Conforto proof of the characterization of rational double planes still had a gap, that has been fixed in \cite{Cal}.

Coming back to Theorem \ref{thm:CE}, its first correct proof is due to Ferretti in \cite{Ferretti}. This is essentially exposed in Enriques--Chisini's book  \cite[vol.~III, \S21, pp.~187--190]{EC}. However, at p.~190, at the end of the proof of Theorem \ref{thm:CE} (with the correct statement), Enriques and Chisini insisted on the wrong statement that the irreducibility assumption on $C$ can be weakened to reducedness. A possible explanation for this mistake may reside in the fact that the numerical properties of the multiplicities of the curve $C$ (essentially Noether's inequality) stay the same even if it is reducible. However the condition that three points $P_0,P_1,P_2$ of the highest multiplicities are not aligned if $m_0+m_1+m_2>d$, where $d=\deg(C)$, may fail for reducible curves. Moreover, even if one can apply a Cremona transformation decreasing $d$, one or more components of $C$ could be contracted to points: in that case, if one then applies another Cremona transformation based at those points, such components reappear causing the argument to become circular.
The same considerations suggest that one cannot simply proceed by induction on the number of components of $C$.

Few years after Ferretti's work, in 1907, Marletta  gave in \cite{Marletta} a similar proof of Theorem \ref{thm:CE}, by showing the following:

\begin{theorem}[Marletta]\label{thm:marl} 
A curve of Cremona minimal degree $d>1$, with the point of maximal multiplicity $m_0>d/3$, has non-vanishing adjoint linear system of index $i$, with $i=[(d-m_0)/2]$, where $[x]$ denotes the largest integer smaller than or equal to $x$.
\end{theorem}

Also Ferretti in \cite{Ferretti} had given similar interesting results regarding the adjoint linear systems to Cremona minimal models.

Though Theorem \ref{thm:CE} is nowadays called Coolidge's Theorem, its proof  in Coolidge's book \cite[pp.~396--398]{Coolidge} contains the same gap pointed out by Segre for Noether's argument. 
It is strange how careless Coolidge has been in references:  the only one he gives is to a paper of Franciosi of 1918. It is also very strange that 
Coolidge's wrong proof has been repeated \emph{verbatim} in the paper \cite{KumarMurthy} by Kumar and Murthy, who however gave a correct proof of Theorem \ref{thm:KumarMurthy} with different methods.

Regarding the minimal degree problem, it seems that Giuseppe Jung in 1889 (see \cite{Jung2}) has been the first one who proved the:

\begin{theorem}[G.~Jung] \label{thm:Jung}
If an irreducible curve $C$ has degree $d$ and maximal multiplicities $m_0\ge m_1 \ge m_2$ with $d\ge m_0+m_1+m_2$, then $C$ has minimal degree.
\end{theorem}

The same statement holds \emph{mutatis mutandis} for irreducible linear systems of plane curves (see \cite[Theorem 2.3]{CC} for a short proof which uses adjoint linear systems).
Theorem \ref{thm:Jung} has been stated by Coolidge in \cite[p.~403]{Coolidge} with the weaker hypothesis $d>m_0+m_1+m_2$, but the proof therein works also in case $d=m_0+m_1+m_2$.
 
If an irreducible curve $C$ has Cremona minimal degree $d$ and maximal multiplicities $m_0 \ge m_1 \ge m_2$ with $m_0+m_1+m_2>d$, one sees that the corresponding points $P_0,P_1,P_2$ are infinitely near, namely 
$P_0\in\PP^2$, $P_1$ is infinitely near to $P_0$ and either\\
\begin{inparaenum}[$\bullet$]
\item $P_2$ is infinitely near to $P_0$ (in a different direction with respect to $P_1$), or\\
\item $P_2$ is infinitely near to $P_1$ and $P_2$ is satellite to $P_0$.\\
\end{inparaenum}
In both cases, it follows that $m_0>d/2$.

Marletta in \cite{Marletta} gave sufficient conditions on the multiplicities of the singular points of a curve $C$ in order that $C$ has Cremona minimal degree with $d<m_0+m_1+m_2$.

A key ingredient in the study of (linear systems of) plane curves has been the concept of adjoint linear systems.
According to Enriques and Chisini in \cite[p.~191]{EC}, they have been originally introduced by Brill and Noether in 1873 for the study of linear series on curves (see \cite {BN}).  Their invariance with respect to Cremona transformations has been firstly used by Kantor in 1883 (but published only in 1891, see \cite  {Ka}) and then by Castelnuovo in 1891 (see \cite {Cas2}).

\section{Preliminaries and notation} \label{S:notation}

\subsection{Adjoint linear systems} Let $C$ be a reduced plane curve.
Let $f\colon S\to\PP^2$ be a birational morphism which resolves
the singularities of $C$ and denote by $\tilde C$ the strict transform of $C$ on $S$.
For any pair of integers $n\ge 1$ and $m\ge n$, we set
\begin{equation*}\label{eq:adm}
\ad_{n,m}(C)=f_*(|n\tilde C+mK_S|).
\end{equation*}
and 
\[
\ad_{m}(C):=\ad_{1,m}(C),\quad \text{so that}\quad \ad_{n,m}(C)=\ad_{m}(nC).
\]
We call $\ad_{n,m}(C)$ the  \emph{$(n,m)$--adjoint linear system to $C$}, and $\ad_{m}(C)$ simply 
the \emph{$m$--adjoint linear system to $C$}, or adjoint linear system \emph{of index $m$} to $C$.

\begin{remark} The reason why we assume $m\ge n$ is that the definition  of $(n,m)$--adjoint systems is independent of the morphism $f$ only in this case (the reader will easily check this).

Moreover, $\ad_{n,m}(C)=f_*(|n\tilde C'+mK_S|)$, with $\tilde C'=\tilde C+D$, where $D$ is supported  on the exceptional divisors of $f$. 
\end{remark}

A basic role in Cremona geometry  is played by the \emph{$n$--adjoint sequence} $\lbrace \dim(\ad_{n,m}(C))\rbrace_{m\ge n}$ of $C$ (simply called the \emph{adjoint sequence} if $n=1$). A crucial remark is that \emph{adjunction extinguishes}, i.e., $\ad_{n,m}(C)=\emptyset$ for $m\gg 0$, hence adjoint sequences stabilize to $-1$. Therefore we can consider them as finite sequences, ending with the first $-1$ after which it stabilizes. From the results in \cite  {CC} it follows that the adjoint sequence stabilizes to $-1$ as soon as it reaches the value $-1$. 

Let $\gamma\in \Cr_2$ and let $C$ and $C'$ be reduced curves in $\PP^ 2$. We say that \emph{$\gamma$ maps $C$ to $C'$}, if there is a commutative diagram 
\begin{equation}\label{eq:diag}
\xymatrix{
& \tilde S \ar[rd]^\beta \ar[ld]_\alpha \\
\PP^ 2 \ar@{-->}[rr]_\gamma & & \PP^2
}
\end{equation}
with $\alpha, \beta$ birational morphisms, and there is a smooth curve $\tilde C$ on $\tilde S$ such that 
\begin{equation}\label{eq:tr}
\alpha_*(\tilde C)=C, \quad \text{and}\quad \beta_*(\tilde C)=C'. 
\end{equation}

Note that $\gamma$ [resp.\ its inverse] may contract some components of $C$ [resp.\ of $C'$] to points. In particular, $C'$ could be the zero curve, in which case $C$ is said to be  \emph{Cremona contractible} or \emph{Cr--contractible}.

As recalled in \S \ref {S:history},  
the following lemma is basically due to S. Kantor:

\begin {lemma}[Kantor] \label{lem:Kantor} If $C$ is a reduced plane curve, for all integers $n\ge 1$ the $n$--adjoint sequence of $C$ is a Cremona invariant.
\end{lemma}

\begin {proof} It follows from diagram \eqref {eq:diag}, from \eqref {eq:tr} and from 
\[
\ad_{n,m}(C)=\alpha_*(|n\tilde C+mK_S|), \quad \ad_{n,m}(C')=\beta_*(|n\tilde C+mK_S|).
\] \end{proof}

\begin{remark}\label{rem:van} In \cite[\S 4]{CC} there is a proof of Kantor's Lemma, under a useless restrictive hypothesis.

Though irrelevant for us, it should be noted that, strictly speaking, the adjoint systems themselves, \emph{are not} Cremona invariant, due to the possible existence for them of (exceptional) fixed components which can be contracted by a Cremona transformation. 
\end{remark}

\subsection{Pairs} Let $(S,D)$ be a \emph{pair}, namely $D$ is a reduced curve on a smooth projective surface $S$.
For any non--negative integer $m$, the \emph{$m$--log plurigenus} of $(S,D)$ is
\[
P_m(S,D):=h^0(S, \cO_S(m(D+K_S)).
\]
If $P_m(S,D)=0$ for all $m\ge1$, then one says that the \emph{log Kodaira dimension} of the pair $(S,D)$ is $\kod(S,D)=-\infty$. Otherwise
\[
\kod(S,D)=\max\left\{ \dim \left( {\rm Im} \left( \phi_{|m(D+K_S)|} \right) \right)\right\}
\]
where $\phi_{|m(D+K_S)|} $ is the rational map determined by the linear system $|m(D+K_S)|$, whenever this is not empty. 

A pair $(S,D)$ is said to be \emph{contractible} if there exists a birational map $\psi\colon S\rto S'$ such that $D$ is contracted by $\psi$ to a union of points, namely $\psi$ is constant on any irreducible component of $D$.

\begin{remark}\label{rem:contractible1} Assume $S$ rational. 
If $\psi:S \to S'$ is a birational morphism which contracts $D$, then $D$ is contained in the exceptional locus of $\psi$, therefore all connected components of $D$ have arithmetic genus 0, in particular they have normal crossings. Moreover, $\kod(S,D)=-\infty$.

If $(S,D)$ is contractible, then there is a resolution of the indeterminacies of $\psi$, i.e., a commutative diagram
\[
\xymatrix{
& \tilde S \ar[rd]^\beta \ar[ld]_\alpha \\
S \ar@{-->}[rr]_\psi & & S'
}
\]
where $\alpha$ and $\beta$ are birational morphisms.
If $\tilde D$ is the strict transform of $D$ via $\alpha$, then $\tilde D$ is contracted to a union of points by the morphism $\beta$, hence $\kod(\tilde S,\tilde D)=-\infty$. 
\end{remark}

Let $(S,D)$ and $(S',D')$ be pairs. We will say that $(S,D)$ and $(S',D')$ are \emph{birationally equivalent} if there is a birational map $\phi: S\dasharrow S'$ such that $\phi$ [resp.\ $\phi^ {-1}$] does not contract any irreducible component of $D$ [resp.\ of $D'$] and the image of $D$ via $\phi$ is $D'$ (hence the image of $D'$ via $\phi^ {-1}$ is $D$).

With this definition we immediately have:

\begin{lemma}
If $(S,D)$ is birationally equivalent to $(S',D')$ and $(S',D')$ is contractible, then $(S,D)$ is contractible too.
\end{lemma}

\begin{remark}
If a reduced plane curve $C$ is Cr-contractible, that is condition $(a)$ in the Introduction, then, setting $(S,\tilde C)$ as at beginning of the section, one has $\kod(S,\tilde C)=-\infty$, which is condition $(b)$ in the Introduction.

Since $\tilde C$ is effective, condition $(b)$ is actually equivalent to
\begin{equation*}\label{eq:ad_mn}
\ad_{n,m}(C)=\emptyset,
\qquad
\text{ for each $m\ge n \ge1$.}
\end{equation*}
In particular $(b)$ implies
\begin{equation}\label{eq:ad_m}
\ad_{m}(C)=\emptyset,
\qquad
\text{ for each $m\ge1$,}
\end{equation}
which is condition $(c)$ in the Introduction.

Condition $(c)$ trivially implies  $(d)$.
Let us see that  $(d)$ is equivalent to  $(d')$.
Indeed, $P_1(S,\tilde C)=0$ is equivalent to $|\tilde C+K_S|=\emptyset$. Then, by adjunction, all irreducible components of $\tilde C$ are smooth rational curves. Thus, $2\tilde C+2K_S$ intersects all components of  $\tilde C$ negatively, so that $\tilde C$ is in the fixed part of  $|2\tilde C+2K_S|$, if this is not empty. In conclusion, one has $P_2(S,\tilde C)=\dim (|2\tilde C+2K_S|)+1=0$ if and only if $|\tilde C+2K_S|=\emptyset$, that is $\ad_2(C)=\emptyset$.
\end{remark}

\subsection{Generalities on union of lines}
In this paper we will study curves $C$ which are unions of distinct lines and have vanishing adjoints, i.e., such that \eqref{eq:ad_m} holds. Then we will see which of them have $\kod(S,\tilde C)=-\infty$ and which are Cr-contractible. We will now prepare the territory for this. 

Let $C$ be a reduced plane curve with  $d=\deg(C)$. If $C$ is singular, let $m_0\ge m_1 \ge \cdots \ge m_r\ge 2$ be the  multiplicities of the singular points $P_0,\ldots,P_r$ of $C$, which can be \emph{proper} or \emph{infinitely near}. We set  $m_{i}=1$ if $i>r$. If $C$ is smooth, we assume $m_0=1$ and $P_0$ is a general point of $C$, and $m_i=1$ for $i>0$. 
By the \emph{proximity inequality} (see \cite {CC} as a general reference for these matters and for notation), we may and will assume that $P_i \infnear{} P_j$ (i.e.\ $P_i$ is infinitely near to $P_j$) implies $i > j$.
Therefore $P_0$ is proper and either $P_1$ is also proper  or $P_1 \infnear[1] P_0$ (i.e.\ $P_1$ is infinitely near of order 1 to $P_0$).

As announced in the Introduction, we will use the notation $(d;m_0,m_1,\ldots, m_r)$ to denote a linear system of plane curves of degree $d$ with assigned base points $P_0,\ldots,P_r$ with respective multiplicities at least $m_0,\ldots, m_r$. Hence we may write $C\in (d;m_0,m_1,\ldots, m_r)$.
We will write $C\cong(d;m_0,m_1,\ldots, m_r)$ whenever $C$ has multiplicity exactly $m_i$ at $P_i$, $i=1,\ldots,r$.
If $C\cong (d;m_0,m_1,\ldots, m_r)$, then ${\rm ad}_{n,m}(C)=(nd-3m;nm_0-m,\ldots,nm_q-m)$, where $q$ is the maximum such that $nm_q>m$.

\begin{lemma}\label{lem:1} In the above setting, one has  $d-m_0 \ge 0$ with equality if and only if $C$ consists of $d$ lines in the pencil of centre $P_0$, in which case $C$ is Cr-contractible.\end{lemma}

\begin{proof} Only the last assertion needs to be justified. Let $d=2\ell+\epsilon$, with $\epsilon\in\{0,1\}$. 
Consider the de Jonqui\`eres transformation of degree $\ell+1+\epsilon$ centered at $P_0$, with multiplicity $\ell+\epsilon$, at  $d$ simple base points, each one general on a component of $C$, and, in addition  if $\epsilon=1$, at further general simple base point. This transformation contracts $C$ to $d$ distinct points of the plane.
\end{proof}

Set 
\begin{equation}\label{eq:set}
d-m_0=2h+\varepsilon,\quad \text{with}\quad \varepsilon\in\{0,1\}.
\end{equation}

\begin{lemma}\label{lem:2} In the above setting, if $d=m_0+1$, then  $C$ consists of  $\ell\ge 0$ lines in the pencil of centre $P_0$, plus an irreducible curve $C'$ of degree $d'=d-\ell$, with $P_0$ of multiplicity $d'-1$ and all points of $C$ off $P_0$ are non--singular. Then $C$ is Cr-contractible.\end{lemma}

\begin{proof}  Only the last assertion needs to be justified. Assume $d'>1$.
Then the curve $C'$ is mapped to a line $L$ by a de Jonqui\`eres transformation of degree $d'$ centered at $P_0$, with multiplicity $d'-1$, and $2d'-2$ general simple base points of $C'$.  The curve $C$ is then transformed to
a curve $\bar C$ consisting of $\ell$ lines $L_1,\ldots,L_\ell$ in a pencil of centre a point $\bar P_0$, plus the line $L$ not passing through $\bar P_0$. So we are reduced to the case $d'=1$.

If $\ell=0$, we finish by contracting $L$ to a point with a quadratic transformation based at three points two of which lie on $L$. Similarly if $\ell=1$. So we may assume $\ell\ge 2$. 
Set then $\mu=(\ell+2)/2$ if $\ell$ is even or $\mu=(\ell+3)/2$ if $\ell$ is odd and
note that $\ell\ge\mu\ge 2$.
The de Jonqui\`eres transformation of degree $\mu$ centered at $\bar P_0$, with multiplicity $\mu-1$, plus  $\mu$ simple base points at $L\cap L_1$, $L\cap L_2$, \dots, $L\cap L_\mu$, plus $\ell-\mu$ simple base points each general  on one of the lines $L_{\mu+1}$, $L_{\mu+2}$, \dots, $L_{\ell}$, and further $2\mu-2-\ell$ simple general base points,  contracts $\bar C$ to $\ell+1$ distinct points of the plane.
\end{proof}

From now on, we may and will assume $h\ge 1$.

\begin{lemma} \label {lem:3} In the above setting, if \eqref {eq:ad_m} holds,  one has:\\
\begin{inparaenum}[(i)]
\item $m_0>h$, or equivalently $m_0>d/3$;\\
\item $m_{i} >h$ for $1\le i\le 2$;\\
\item $m_0+m_1+m_2\ge d+1$.
\end{inparaenum}
\end{lemma}

\begin{proof} (i) The equivalence between  $m_0>h$ and $m_0>d/3$ is clear. Let us prove that $m_0>d/3$. The assertion is trivial if $d<3$, so we assume $d\ge3$.
Suppose $m_0\le d/3$. Then  ${\rm ad}_{[d/3]}(C)$ is the complete linear systems of curves of degree $d-3\left[\frac{d}{3}\right]\ge 0$ which is not empty,   contradicting \eqref{eq:ad_m}.

(ii) If $m_2\le h$, then 
$ {\rm ad}_h(C)=(m_0-h+\varepsilon; m_0-h, \overline {m_1-h})$, where $\overline {m_1-h}=\max\{{m_1-h},0\}$, is not empty, contradicting \eqref{eq:ad_m}. 

(iii) Follows from (i) and (ii). \end{proof}

In the rest of this section, we consider the case that $C$ is a reduced union of $d$ lines.
Since $m_i$ is the number of lines passing through $P_i$, for $i=0,\ldots,r$, one has
\begin{equation}\label{upper}
m_0+m_1+m_2 \le d+3.
\end{equation}
See Lemma \ref{lem:comb1} below for the description of the case where equality occurs.

If $C\cong (d;m_0,\ldots,m_r)$ is a reduced union of lines, then we say that  $(d;m_0,\ldots,m_r)$ is the \emph{type} of $C$. Moreover, we will say that two reduced union of lines
\[
C=L_1\cup L_2 \cup \cdots \cup L_d
\qquad
\text{ and }
\qquad
D=R_1\cup R_2 \cup \cdots \cup R_d
\]
have the same \emph{configuration} if there exists a permutation $\sigma$ of $\{1,\ldots,d\}$ such that
\[
L_{i_1} \cap L_{i_2} \cap \cdots \cap L_{i_k}\ne\emptyset
\ \Longleftrightarrow\ 
R_{\sigma(i_1)} \cap R_{\sigma(i_2)} \cap \cdots \cap R_{\sigma(i_k)}\ne\emptyset,
\quad\text{ for each $i_1,\ldots,i_k$.}
\]

\begin{remark}\label{typenotconf}
Clearly, two reduced unions of lines with the same configuration are also of the same type, but, in general, the type does not uniquely determine the configuration.
For example, if $C$ is the reduced union of 6 lines with two triple points, i.e., the type of $C$ is $(6;3^2,2^9)$, then there are exactly two configurations of this type, according to the following possibilities: either the line passing through the triple points is a component of $C$ or it is not.
\end{remark}

We will denote a configuration of a reduced union of lines $C=L_1\cup\cdots\cup L_d$ as follows:
\begin{equation} \label{configuration}
(d;\{a_{0,1},a_{0,2},\ldots,a_{0,m_0}\}, \{a_{1,1},a_{1,2},\ldots,a_{1,m_1}\}, \ldots, \{a_{s,1},a_{s,2},\ldots,a_{s,m_s}\})
\end{equation}
where $m_s\geq3$ and $P_0=L_{a_{0,1}}\cap L_{a_{0,2}} \cap \cdots \cap L_{a_{0,m_0}}$ is a point of multiplicity $m_0$, 
$P_1=L_{a_{1,1}}\cap L_{a_{1,2}} \cap \cdots \cap L_{a_{1,m_1}}$ is a point of multiplicity $m_1$, 
and so on, for all points of multiplicity $\geq3$.
In other words, \eqref{configuration} lists the singular points of $C$ of multiplicity $m\geqslant3$ and the lines containing each of them.
In particular, we list the singular points according to their multiplicities in non--increasing order, and, among the points of the same multiplicity, we will usually list them in lexicographical order with respect to the given ordering of the lines.

\begin{remark}\label{double}
In \eqref{configuration} one does not need to list the nodes, i.e., the double points of $C$. Indeed, $P_{i,j}=L_i\cap L_j$ is a node of $C$ if and only if there exist no $k,h,l\in \lbrace 1,\ldots,r\rbrace$ such that $a_{k,h}=i$ and $a_{k,l}=j$.
The number of nodes of $C$ is
\[
\binom{d}{2}-\sum_{i=0}^s \frac{m_i(m_i-1)}{2}\,.
\]
\end{remark}

\begin{example}
If $C$ is a reduced union of 6 lines of type $(6;3^2,2^9)$ such as in Remark \ref{typenotconf},  the two possible configurations are $(6,\{1,2,3\},\{1,4,5\})$ and $(6,\{1,2,3\},\{4,5,6\})$. The former means that the triple points of $C$ are $P_0=L_1\cap L_2\cap L_3$ and $P_1=L_1\cap L_4\cap L_5$, so that the nodes are $P_{1,6}=L_1\cap L_6$, $P_{2,4}$, $P_{2,5}$, $P_{2,6}$, $P_{3,4}$, $P_{3,5}$, $P_{3,6}$, $P_{4,6}$, $P_{5,6}$, cf.\ Remark \ref{double}.
The latter configuration means that the triple points of $C$ are $P_0=L_1\cap L_2\cap L_3$ and $P_1=L_4\cap L_5\cap L_6$, so that the nodes are $P_{1,4}=L_1\cap L_4$, $P_{1,5}$, $P_{1,6}$, $P_{2,4}$, $P_{2,5}$, $P_{2,6}$, $P_{3,4}$, $P_{3,5}$, $P_{3,6}$, cf.\ Remark \ref{double}.
\end{example}

\begin{remark}
We will see later in Remark \ref{deg9} that different configurations of the same type may behave quite differently with respect to the adjoint linear systems and Cremona contractibility.

Furthermore, two reduced union of lines with the same configuration are not necessarily projectively equivalent.
For example, if the type of $C$ is $(4;4)$, then there is only one possible configuration, but the isomorphism classes of four lines passing through a point depend on one parameter.
\end{remark}

\section{Configurations of lines with  vanishing adjoints}\label{s:adj}

In this section we classify reduced plane curves $C$ which are union of $d\ge 12$ distinct lines and such that \eqref{eq:ad_m} holds, i.e., with vanishing adjoint linear systems.

\subsection{Basics}\label {ssec:gen}  We keep the notation 
introduced above, including \eqref {eq:set}. The degree $d$ of $C$ is the number of its components and the singular points of $C$ are all proper. We will assume that \eqref {eq:ad_m} holds.

Set $d=3\delta +\eta$, with $0\le \eta\le 2$. By Lemma \ref {lem:3} we have $m_0>\delta$, and we set $m_0=\delta+\mu$, with $\mu\ge 1$. Set $\mu=2\nu+\tau$, with $0\le \tau\le 1$, so that
$d-m_0=2(\delta-\nu)+(\eta-\tau)$, thus:\\
\begin{inparaenum}
\item  [(i)] $h=\delta -\nu$ and $\varepsilon= \eta-\tau$, unless either\\
\item [(ii)]  $\eta=0, \tau=1$, in which case $h=\delta-\nu-1$ and $\varepsilon=1$, or\\
\item [(iii)] $\eta=2, \tau=0$, in which case $h=\delta-\nu+1$ and  $\varepsilon=0$.
\end{inparaenum}

We set
\[
m:=m_0+m_1+m_2.
\]
By \eqref {upper} one has $m\le d+3=3\delta+\eta+3$.
Since $m_1,m_2\ge h+1$ by  (ii) of Lemma \ref {lem:3}, we have:\\
\begin{inparaenum}[$\bullet$]
\item  $3\delta+\tau+\varepsilon+3\ge m\ge 3\delta+\tau+2$,  in case (i);\\
\item  $3\delta+3\ge m\ge  3\delta+1$, in case (ii);\\
\item   $3\delta+5\ge m\ge  3\delta+4$, in case (iii);\\
\end{inparaenum}
thus the interval in which $m$ lies is $[d+2-\varepsilon, d+3]$ and its length is
$\varepsilon+1\in \{1,2\}$, hence  $d+1\le m\le d+3$.
The following table shows the possible values of $m_1$ and $m_2$:
\begin{equation}\label{table}
\begin{array}{|c|c|c|c|c|c|c|c|c|c|c}
 \hline m_1 & m_2  &\varepsilon  & m & \text{possible cases}\\ \hline\hline
 h+1 & h+1 & 0,1  & d+2-\varepsilon & \text{(i)-(ii)-(iii)} \\ \hline
 h+2 & h+1 & 0,1 & d+3-\varepsilon  & \text{(i)-(ii)-(iii)}  \\ \hline
 h+2 & h+2 & 1  & d+3  & \text{(i)-(ii)} \\ \hline 
 h+3 &h+1 & 1& d+3  & \text{(i)-(ii)} \\ \hline 
\end{array}
\end{equation}

We will use the following notation
\[
m_2=\cdots= m_k, \quad m_2-1=m_{k+1}=\cdots=m_{k+l}>m_{k+l+1}
\]
where $k\ge2$ and $l\ge0$.
It will be essential for us the consideration of 
\[{\rm ad}_h(C)=(m_0-h+\varepsilon; m_0-h, m_1-h, (m_2-h)^{k-1},    (m_2-h-1)^{l},  \ldots )\]
which has to be empty and we note that $1\le m_1-h\le 3$ whereas $1\le m_2-h\le 2$.

The proofs of the following lemmas are elementary and can be left to the reader.

\begin{lemma}\label {lem:comb1} In the above setting, if $m=d+3$, then:\\
\begin{inparaenum}[$\bullet$]
\item $P_0,P_1,P_2$  are not collinear  and the sides of the triangle with vertices $P_0,P_1,P_2$ are components of $C$; \\
\item all components of $C$ pass through one of the points $P_0,P_1,P_2$;\\
\item the remaining singular points of $C$ have multiplicity at most 3.
\end{inparaenum}
\end{lemma}

\begin{lemma}\label {lem:comb2} In the above setting, if $m=d+2$, then:\\
\begin{inparaenum}
\item [($\alpha$)] either $P_0,P_1,P_2$  are collinear and the line joining them belongs to $C$, in which case 
all components of $C$ pass through one of the points $P_0,P_1,P_2$ and  the remaining singular points of $C$ have multiplicity at most 3;\\
\item [($\beta$)] or  $P_0,P_1,P_2$  are not collinear and    the sides of the triangle with vertices $P_0,P_1,P_2$ are components of $C$, in which case 
all components of $C$ but one pass through one of the points $P_0,P_1,P_2$, the remaining singular points of $C$  have multiplicity at most 4 and there are at most two of them with multiplicity 4;\\
\item [($\gamma$)] or  $P_0,P_1,P_2$  are not collinear and two of the three the sides of the triangle with vertices $P_0,P_1,P_2$ are components of $C$, in which case 
all components of $C$ pass through one of the points $P_0,P_1,P_2$ and  the remaining singular points of $C$ have multiplicity at most 3.
\end{inparaenum}
\end{lemma}

\begin{lemma}\label {lem:comb3} In the above setting, if $m=d+1$, then:\\
\begin{inparaenum}
\item [($\alpha'$)] either $P_0,P_1,P_2$  are collinear and the line joining them belongs to $C$, in which case 
all components of $C$ but one pass through one of the points $P_0,P_1,P_2$ and  the remaining singular points of $C$  have multiplicity at most 4;\\
\item [($\beta'$)] or  $P_0,P_1,P_2$  are not collinear and  the  sides of the triangle with vertices $P_0,P_1,P_2$ are components of $C$, in which case 
all components of $C$ but two pass through one of the points $P_0,P_1,P_2$,  the remaining singular points of $C$ have multiplicity at most 5 and there is at most one of them with multiplicity 5;\\
\item [($\gamma'$)] or  $P_0,P_1,P_2$  are not collinear and two of the three the sides of the triangle with vertices $P_0,P_1,P_2$ are components of $C$, in which case 
all components of $C$ but one pass through one of the points $P_0,P_1,P_2$ and   the remaining singular points of $C$  have multiplicity at most 4;\\
\item [($\delta'$)] or  $P_0,P_1,P_2$  are not collinear and only one side of the triangle with vertices $P_0,P_1,P_2$ is a component of $C$, in which case 
all components of $C$ pass through one of the points $P_0,P_1,P_2$ and   the remaining singular points of $C$  have multiplicity at most 3.
\end{inparaenum}

\end{lemma}

\subsection{The case $m$ maximal}
Here we treat the case $m=d+3$, in which Lemma \ref {lem:comb1} applies. 
\subsubsection{The subcase $\varepsilon=0$}  We are either in case (i) or (iii) and in table \eqref {table} the second row  occurs, hence $m_0\ge m_1=h+2$.
Thus
\[{\rm ad}_h(C)=(m_0-h; m_0-h, 2, 1^{k-1}).\] 

\begin {lemma}\label {lem:num1} Assume \eqref {eq:ad_m} holds, $m=d+3$, $d\ge 12$, $h\ge 1$, $\varepsilon =0$, then $C\cong(d;d-2, 3, 2^ {2(d-3)})$.
\end{lemma}
\begin{proof}   
We claim that $m_0-h>2$. Otherwise  $m_0-h=2$ hence $d=3h+2$. One has
\[
{\rm ad}_{h-1}(C)=(5; 3^2, 2^{k-1}, 1^{l})
\] 
which has to be empty.
Then either $k\ge3$ or $l\ge1$ (recall that $k\ge2$).
Taking into account the last item of Lemma \ref {lem:comb1},  we see that $h\le 3$, hence $d\le 11$, a contradiction.

Then $m_0-h\ge 3$ and since ${\rm ad}_{h}(C)$ is empty, one has $k\ge 3$.  The last item of Lemma \ref {lem:comb1}
implies  $h\le 2$. If $h=2$, then $m_1=4$, $m_2=3$,  and Lemma \ref {lem:comb1} again yields  $k\in\{3,4\}$. 
Emptiness of  ${\rm ad}_2(C)=(d-6;d-6, 2, 1^{k-1})$ implies $d\le 10$, a contradiction.  If $h=1$ we find the assertion. 
\end{proof}

\begin{prop}\label{prop:class1} If $C$ is a union of lines and $C\cong(d;d-2, 3, 2^ {2(d-3)})$, then $C$ is Cr-contractible.
\end{prop}
\begin{proof} The assertion is clear for $d=3,4$ by Lemmas \ref{lem:1} and \ref{lem:2}, so we may assume $d\ge5$ and we proceed by induction on $d$. 
 Consider the two lines $L_1,L_2$ through $P_1$ not passing through $P_0$, and consider two more lines $L_3,L_4$ through $P_0$ and not through $P_1$. Consider the intersection points $P_{1,3}=L_1\cap L_3$, $P_{2,4}=L_2\cap L_4$. Make a quadratic transformation based at $P_0, P_{1,3},P_{2,4}$.  This maps $C$ to a union of lines $C'\cong(d-2;d-4, 3, 2^ {2(d-4)})$: the lines $L_3, L_4$ have been contracted to two points of the transforms of the lines $L_1,L_2$ which do not lie on any other component of $C'$. The assertion follows by induction. \end{proof}

\subsubsection{The subcase $\varepsilon=1$}  We are now either in case (i) or (ii) and in table \ref {table} the last two rows occur.  Thus, either
\begin{align}\label{eq:case1}
 m_0\ge m_1=h+3, \,\, m_2=h+1,\,\, &{\rm hence }\,\,\,  {\rm ad}_h(C)=(m_0-h+1; m_0-h, 3, 1^{k-1}),\text{ or}
\\
\label{eq:case2}
m_0\ge m_1=m_2=h+2 ,\,\, &{\rm hence }\,\,\,  {\rm ad}_h(C)=(m_0-h+1; m_0-h,  2^{k}, 1^{l}).
\end{align}

\begin {lemma}\label {lem:num2} If  \eqref {eq:ad_m} holds, $m=d+3$, $d\ge 11$, $h\ge 1$ and $\varepsilon =1$, then:
\begin{enumerate}[(a)]
\item  either $C\cong(d;d-3, 4, 2^ {3(d-4)})$,
\item  or $C\cong(d;d-3, 3^ 2, 2^ {3(d-4)})$ or $C\cong(d;d-3, 3^ 3, 2^ {3(d-5)})$. 
\end{enumerate}
\end{lemma}
\begin{proof}  Note that in case \eqref {eq:case1} one has $m_0-h\ge  3$  and in case \eqref {eq:case2} one has  $m_0-h\ge  2$. So,  
to make ${\rm ad}_h(C)$ empty we must have:\\
\begin{inparaenum}[$\bullet$]
\item  $k\ge 5$ in case  \eqref {eq:case1} and\\
\item  either $k\ge 3$, or $k=2$ and $l\ge 1$, in case \eqref  {eq:case2}.  
\end{inparaenum}

In case  \eqref {eq:case1}, the last item of Lemma   \ref {lem:comb1} yields $h\le 2$. If $h=1$, we are in case (a). If $h=2$, then $m_2=3$ implies $k\le 5$. Then emptiness of ${\rm ad}_h(C)=(d-6;d-7,3,1^{k-1})$  requires $d\le 10$, contrary to the assumption. 

In case \eqref {eq:case2},   Lemma   \ref {lem:comb1} yields again $h\le 2$. 
If $h=1$, then  $m_2=3$ implies $k\le 3$ and one has the two cases in (b). 
If $h=2$, then $k=2$ and $l\le 4$ and the emptiness of ${\rm ad}_h(C)=(d-6;d-7,2^2,1^l)$ implies $d\le 10$, contrary to the assumption. \end{proof}

\begin{prop}\label{prop:class2} If $C$ is either  as in (a) of Lemma \ref {lem:num2} and $d\ge 11$ or as in (b) and $d\ge 12$, then $\ad_{2,3}(C)\ne\emptyset$ hence $C$ is not Cr-contractible.
\end{prop}
\begin{proof} In case (a) one has that the fixed part of
 ${\rm ad}_{2,3}(C)=(2d-9;2d-9, 5, 1^ {3(d-4)})$ consists of the $d-3$ components of $C$ through $P_0$, the one joining $P_0$ and $P_1$ with multiplicity 5, and the movable part by $d-10$ general lines through $P_0$.   
 
 In cases (b), one has ${\rm ad}_{2,3}(C)=(2d-9;2d-9, 3^2, 1^ {3(d-4)})$ and ${\rm ad}_{2,3}(C)=(2d-9;2d-9, 3^3, 1^ {3(d-5)})$ respectively. The latter is non--empty:  its fixed part consists of the $d-3$ components of $C$ through $P_0$, the ones joining $P_0$ with $P_1,P_2,P_3$ with multiplicity 3, and the movable part by $d-12$ general lines through $P_0$.   The former is also non--empty:  its fixed part consists of the $d-3$ components of $C$ through $P_0$, the ones joining $P_0$ with $P_1,P_2$  with multiplicity 3, plus the line joining  $P_0$ with the intersection of the two distinct lines in $C$ containing $P_1,P_2$, 
 and the movable part by $d-11$ general lines through $P_0$.    \end{proof}

\subsection{The case $m$ minimal}
Now we treat the different extremal case in which $m=d+2-\varepsilon$, hence $m_1=m_2=h+1$ (first line of \eqref {table}). 

\subsubsection{The subcase $\varepsilon=0$}  \label {sss: epsilon0} We are either in case (i) or (iii) and we have $m=d+2$ and 
\[{\rm ad}_h(C)=(m_0-h; m_0-h, 1^{k}).\]

\begin {lemma}\label {lem:num3} Assume \eqref {eq:ad_m} holds, $m=d+2$, $d\ge 12$, $h\ge 1$, $\varepsilon =0$, then $C\cong(d;d-2, 2^ {2d-3})$.
\end{lemma}
\begin{proof}  
Assume first  $m_0-h=1$ hence $d=3h+1$. Recall that $k\ge2$.
If $k>2$, then  Lemma \ref {lem:comb2} implies  $h\le 3$, hence $d\le 10$, a contradiction.
If $k=2$, then ${\rm ad}_{h-1}(C)=(4; 2^3, 1^l)$ has to be empty, then we must have $l\ge 6$. Lemma \ref {lem:comb2} implies  $h\le 4$, which can happen only in case $(\beta)$ of Lemma \ref {lem:comb2}, in which case $l\le 2$, a contradiction. If $h\le 3$ then $d\le 10$ and we have a contradiction again. Hence $m_0-h\ge 2$ and ${\rm ad}_{h}(C)=\emptyset$ yields $k>m_0-h\ge 2$.  We discuss separately the various cases in Lemma \ref {lem:comb2}. \medskip

\noindent {\bf Case ($\alpha$).}  We have $h\le 2$. 
If $h=2$, then   $k\le 6$. Since   ${\rm ad}_{2}(C)=(d-6; d-6, 1^k)$
is empty, we have $6\ge k\ge d-5$, a contradiction. If $h=1$ we have the assertion.  \medskip

\noindent {\bf Case ($\beta$).}  We have $h\le 3$. If $h=3$, then ${\rm ad}_{3}(C)=(d-9; d-9, 1^k)$. On the other hand, one  has $k\le 4$, which
leads to $d\le 11$, a contradiction. If $h=2$ then $k\le 6$ and $d\le 11$. So we are left with the case $h=1$, leading to the assertion.\medskip

\noindent {\bf Case ($\gamma$).}  We have  $h\le 2$. 
If $h=2$ one  has $k\le 6$. Emptyness of  ${\rm ad}_{2}(C)$ implies $d\le 11$, a contradiction. 
 If $h=1$ we have the assertion.  \end{proof}

\begin{prop}\label{prop:class2bis} If $C$ is a union of lines and $C\cong(d;d-2, 2^ {2d-3})$,  then $C$ is Cr-contractible.
\end{prop}
\begin{proof}  If $d\le 3$,  the assertion  is trivial. We then argue  by induction on $d$.
Let $L_1,L_2$ be the two lines not passing through $P_0$, and let $L_3, L_4$ be two lines through $P_0$.
Consider the intersection points $P_{1,3}=L_1\cap L_3$, $P_{2,4}=L_2\cap L_4$. Make a quadratic transformation based at $P_0, P_{1,3},P_{2,4}$.  This maps $C$ to a union of lines $C'\cong(d-2;d-4, 2^{2d-7})$: the lines $L_3, L_4$ have been contracted to two points of the transforms of the lines $L_1,L_2$ which do not lie on any other component of $C'$. The assertion follows by induction.
\end{proof}

\subsubsection{The subcase $\varepsilon=1$}   We are either in case (i) or (ii), we have $m=d+1$ and 
\[{\rm ad}_h(C)=(m_0-h+1; m_0-h, 1^{k}).\]

\begin {lemma}\label {lem:num4} Assume \eqref {eq:ad_m} holds, $m=d+1$, $d\ge 12$, $h\ge 1$, $\varepsilon =1$, then $C\cong(d;d-3, 2^ {3(d-2)})$.
\end{lemma}
\begin{proof} 
Since $m_0-h+1\ge 2$ and ${\rm ad}_{h}(C)$ is empty, one has $k\ge 5$. Lemma \ref{lem:comb3} applies.

\medskip
\noindent {\bf Case ($\alpha'$).}  One has $h\le 3$. If $h=3$, then $k\le 5$.
Hence $k=5$ and $\ad_2(C)=(2;1^6)$, but the 6 points do lie on a (reducible) conic, a contradiction.
If $h=2$, then ${\rm ad}_h(C)=(d-6; d-7, 1^{k})$ and to make ${\rm ad}_{h}(C)$ empty we need $k>2(d-6)$. On the other hand one sees that
$k\le 8$ hence $d\le 9$, a contradiction. In case $h=1$ we find the assertion.

\medskip
\noindent {\bf Case ($\beta'$).}  One has $h\le 4$. If $h=4$, then $k\le 3$, a contradiction. If $h=3$, then $k\le 6$, hence $k\in\{5,6\}$. Since $\ad_3(C)=(d-9;d-10,1^k)$, we have $k>2(d-9)$, thus $d\le 11$, a contradiction. If $h=2$ then $k\le 10 $, hence $d\le 10$, a contradiction again. Therefore $h=1$ and we find the assertion.

\medskip
\noindent {\bf Case ($\gamma'$).}  One has  $h\le 3$. If $h=3$, then $k\le 4$, a contradiction. If $h=2$, then $k\le 10$, which forces $d\le 10$, a contradiction. hence $h=1$ and we find the assertion.

\medskip
\noindent {\bf Case ($\delta'$).}  One has  $h\le 2$. If $h=2$, then $k\le 6$, which leads to a contradiction as above. Hence $h=1$ and we find the assertion. \end{proof}

\begin{prop}\label{prop:class3} If $C$ is a union of lines and $C\cong(d;d-3, 2^{3(d-2)})$ with $d\ge 9$,  then $\ad_{2,3}(C)\neq \emptyset$, hence $C$ is not Cr-contractible.
\end{prop}
\begin{proof}
Let $P_1,P_2,P_3$ be the vertices of the triangle formed by the three lines of $C$ not passing through $P_0$.
One has
 ${\rm ad}_{2,3}(C)=(2d-9;2d-9, 1^ {3(d-2)})\ne \emptyset$: its fixed part consists of the $d-3$ components of $C$ through $P_0$ plus the three lines joining $P_0$ with $P_i$, $i=1,2,3$, and the movable part by $d-9$ general lines through $P_0$.  \end{proof}

\subsection{The intermediate case for $m$}
Now we treat the intermediate case in which the length of the interval $[d+2-\varepsilon, d+3]$
 in which $m$ lies is $2$, which forces  
$\varepsilon=1$,  and $m=d+2$ is the intermediate value. Hence we are in the case described in Lemma \ref  {lem:comb2}.
The values of $m_1,m_2$ are given by the second row of table \eqref {table}.  The relevant adjoint is 
\[{\rm ad}_h(C)=(m_0-h+1; m_0-h, 2, 1^{k-1}).\]

\begin {lemma}\label {lem:num5} Assume \eqref {eq:ad_m} holds, $m=d+2$, $d\ge 11$, $h\ge 1$, $\varepsilon =1$,  then $C\cong(d;d-3, 3, 2^ {3(d-3)})$.
\end{lemma}
\begin{proof}  As in the proof of Lemma \ref  {lem:num4}, we have $k\ge 5$. 
Again we make a case by case discussion according to the possibilities listed in Lemma \ref {lem:comb2}.

\medskip
\noindent {\bf Case ($\alpha$).}  One has $h\le 2$. If $h=2$, then $k\le 7$. Since $\ad_2(C)=(d-6;d-7,2,1^{k-1})=\emptyset$, it follows that $d\le10$, a contradiction. If $h=1$ we have the assertion. 

\medskip
\noindent {\bf Case ($\beta$).}  One has $h\le 3$. If $h=3$, then $k\le 4$, a contradiction. If $h=2$, then $k\le 7$, which forces $d\le10$ as above, a contradiction. Hence $h=1$ and we find the assertion.

\medskip
\noindent {\bf Case ($\gamma$).}  One has  $h\le 2$. If $h=2$, then $k\le 7$, which forces $d\le 10$, a contradiction. Hence $h=1$ and we find the assertion \end{proof}

\begin{prop}\label{prop:class4} If $C$ is a union of lines and $C\cong(d;d-3, 3, 2^ {3(d-3)})$ with $d\ge 10$,   then $\ad_{2,3}(C)\neq\emptyset$, hence $C$ is not  Cr-contractible.
\end{prop}
\begin{proof} 
There are two configurations of $C$.
Either $C$ contains the line passing through $P_0$ and $P_1$ or it does not.
In the former case, let $P_2, P_3$ the intersection points of the line not passing through $P_0$ and $P_1$ with the two lines through $P_1$ not passing through $P_0$.
In both cases, one has ${\rm ad}_{2,3}(C)=(2d-9;2d-9, 3, 1^ {3(d-3)})\ne \emptyset$.
Indeed, In the former case, its fixed part consists of the $d-3$ components of $C$ through $P_0$, the one joining $P_0$ with $P_1$ with multiplicity 3, plus the two lines joining $P_0$ with $P_2$ and $P_3$, and the movable part by $d-10$ general lines through $P_0$. 
In the latter case, the fixed part consists of the $d-3$ components of $C$ through $P_0$ plus the line joining $P_0$ with $P_1$ with multiplicity 3, and the movable part by $d-9$ general lines through $P_0$.
\end{proof}

We collect the previous results in the following:

\begin{theorem}\label{thm:d>=12}
Let $C$ be a reduced union of $d\ge12$ lines.
If condition \eqref{eq:ad_m} holds, then $C$ has one of the following types
\begin{align}
\label{d;d-2}
&  (d;d),\ 
&& (d;d-1,2^{d-1}),\ 
&& (d;d-2, 3, 2^{2(d-3)}),\ 
&& (d;d-2, 2^{2d-3}),\ 
\\
\label{d-3a}
&  (d;d-3, 4, 2^{3(d-4)}),\ 
&& (d;d-3, 3^3, 2^{3(d-5)}),\ 
&& (d;d-3, 3^ 2, 2^{3(d-4)}),\ 
&& (d;d-3, 3, 2^ {3(d-3)}),\ 
\\
\label{d-3b}
& (d;d-3, 2^{3(d-2)}). \
\end{align}

The types in \eqref{d;d-2} are Cr-contractible, 
while the types in \eqref{d-3a} and \eqref{d-3b} are not Cr-contractible.
If $S\to\PP^2$ is a birational morphism which resolves the singularities of $C$ and denoting by $\tilde C$ the strict transform of $C$ on $S$,
for the types in \eqref{d;d-2} one has $\kod(S,\tilde C)=-\infty$,
while for the types in \eqref{d-3a} and \eqref{d-3b} one has $P_3(S,\tilde C)>0$, thus $\kod(S,\tilde C)\geqslant0$.

In particular, $C$ is Cr-contractible if and only if $\kod(S,\tilde C)=-\infty$.
\end{theorem}

\begin{proof}
Types $(d;d)$ and $(d;d-1,2^{d-1})$ are Cr-contractible by Lemmas \ref{lem:1} and \ref{lem:2}.
Types $(d;d-2,3,2^{2(d-3)})$ and $(d;d-2,2^{2d-3})$ are Cr-contractible by Propositions \ref{prop:class1} and \ref{prop:class2bis}.

The fact that $P_3(S,\tilde C)>0$ for the types in \eqref{d-3a}, \eqref{d-3b} 
follows since $\ad_{2,3}(C)\neq \emptyset$ for them by Propositions \ref{prop:class2}, \ref{prop:class3} and \ref{prop:class4}.
\end{proof}

\begin{remark}\label{deg9}  It is easy to check that each of the types in \eqref{d;d-2}, \eqref{d-3a}, \eqref{d-3b}
but $(d;d-3,3,2^{3(d-3)})$ has exactly one configuration, whereas the latter has exactly two  configurations, namely 
\[(d;\{4,5,\ldots,d\},\{1,2,3\})\qquad
\text{ and }
\qquad
 (d;\{4,5,\ldots,d\},\{2,3,4\})
 \]
 (see the proof of Propositon \ref {prop:class4}). By Proposition \ref{prop:class4}, both are not Cr--contractible if $d\ge 10$. 

It is interesting to notice that instead, for $d=9$, the two configurations above
behave quite differently with respect to Cr-contractibility. Indeed,  the latter one is not Cr-contractible, whereas we will see in a moment that the former is instead Cr-contractible.
Both configurations have vanishing adjoint linear systems, but the former one has also $\ad_{n,m}(C)=\emptyset$ for every $m\ge n\ge 1$, whereas the latter one has, as we saw, $\ad_{2,3}(C)\ne\emptyset$.
\end{remark}

\begin{lemma}\label{conic}
Let $C$ be a reduced plane curve. Suppose that there is a Cremona transformation $\gamma$ such that $\gamma(C)=B\cup Z$, where $B$ is either a line or a conic and $Z$ is either $\emptyset$ or a union of points.
Then $C$ is Cr-contractible.
\end{lemma}

\begin{proof}
Suppose that $B$ is a line. Choose two general points $Q_1,Q_2\in B$ and a general point $Q_3\in\PP^2$. Let $\omega$ be the Cremona quadratic transformation centered at $Q_1,Q_2,Q_3$. Then, $\omega\circ\gamma$ contract $C$ to points.

Suppose that $B$ is a conic. If $B$ is irreducible, choose five general points $Q_1,Q_2,\ldots,Q_5\in B$ and a general point $Q_6\in\PP^2$.
If $B$ is a reducible conic, union of the two lines $R_1$ and $R_2$, choose two general points $Q_1,Q_2\in R_1$, three general points $Q_3,Q_4,Q_5\in R_2$, and a general point $Q_6\in\PP^2$.
In both cases, the Cremona map $\omega$ defined by the homaloidal net $|4L-2Q_1-2Q_2-2Q_3-Q_4-Q_5-Q_6|$ is such that $\omega\circ\gamma$ contracts $C$ to points.
\end{proof}

\begin{prop}
Let $C\cong (9;6,3,2^{18})$ be the configuration
\[
(9;\{1,2,3,4,5,6\},\{1,7,8\}). 
\]
Then $C$ is Cr-contractible.
\end{prop}

\begin{proof}
Let $P_0$ be the point of multiplicity 6 and $P_1$ be the triple point.
Denote, as usual, $P_{i,j}=L_i\cap L_j$, for $i\ne j$.

Take the de Jonqui\`eres map $\gamma_1$ defined by the homaloidal net
$|4L-3P_0-P_1-P_{4,7}-P_{5,8}-P_{6,9}-P_{7,9}-P_{8,9}|$.
Note that $\gamma_1$ contracts $L_1,L_4,L_5,L_6$ to points and maps the other 5 lines to a pentagon. Setting $\bar L_i=\gamma_1(L_i)$, $i=2,3,7,8,9$, and $\bar P_{i,j}=\bar L_i\cap \bar L_j$, for $i\ne j$, one sees that $\gamma_1(L_4)=\bar P_{8,9}$, $\gamma_1(L_5)=\bar P_{7,9}$, $\gamma_1(L_6)=\bar P_{7,8}$ and that $\gamma_1(L_1)$ is a point lying on $\bar L_9$, different from the vertices of the pentagon.

Now the quadratic map $\gamma_2$ centered at $\bar P_{2,8}$, $\bar P_{3,7}$, $\bar P_{3,9}$ contracts $\bar L_3$ to a point and maps the other 4 lines to a quadrilateral.
Setting $\tilde L_i=\gamma_2(\bar L_i)$, for $i=2,7,8,9$, and $\tilde P_{i,j}=\tilde L_i\cap\tilde L_j$, $i\ne j$, one sees that $\gamma_2(\bar L_3)=\tilde P_{2,8}$.

Take then the quadratic map $\gamma_3$ centered at $\tilde P_{2,7}$, $\tilde P_{2,9}$ and a general point $Q_8\in \tilde L_8$.
One sees that $\gamma_3$ contracts $\tilde L_2$ and maps the other 3 lines to a triangle.
Setting $\hat L_i=\gamma_3(\tilde L_i)$, $i=7,8,9$, and $\hat P_{i,j}=\hat L_i\cap\hat L_j$, for $i\ne j$, one sees that $\gamma_3(\tilde L_2)$ is a point lying on $\hat L_8$.

Finally, consider the quadratic map $\gamma_4$ centered at $\hat P_{7,8}$, at the infinitely near point to $\hat P_{7,8}$ in the direction of the line $\hat L_7$ and at a general point $Q_9\in\hat L_9$. One sees that $\gamma_4$ contracts the line $\hat L_7$ to a point and maps the other two lines to a reducible conic.
The choice of the fundamental points of the Cremona maps $\gamma_i$, $i=1,2,3,4$, imply that $\gamma_4\circ\gamma_3\circ\gamma_2\circ\gamma_1$ maps $C$ to a conic, hence $C$ is Cr-contractible by Lemma \ref{conic}.
\end{proof}


\begin{thebibliography}{}

\bibitem {Alberich}
\textsc{Alberich-Carrami\~{n}ana, Maria}:
\textit{Geometry of the plane Cremona maps},
Lecture Notes in Mathematics \textbf{1769},
Springer-Verlag, Heidelberg, 2002.

\bibitem {BN} 
\textsc{Brill, Alexander von, and Noether, Max}: 
\textit{Ueber einen Satz aus der Theorie der algebraischen Functionen}, 
Math. Ann. \textbf{6}  (1873), 351--360. 

\bibitem {Cal1} 
\textsc{Calabri, Alberto}:
\textit{On rational and ruled double planes},
Annali di Matematica Pura ed Applicata \textbf{181} (2002), 365--387.

\bibitem {Cal} 
\textsc{Calabri, Alberto}:
\textit{Rivestimenti del piano, Sulla razionalit\`a dei piani doppi e tripli ciclici},
Centro studi Enriques, Ed. Plus, Pisa, 2006. 

\bibitem {CC} 
\textsc{Calabri, Alberto, and Ciliberto, Ciro}:
\textit{Birational classification of curves on rational surfaces},
Nagoya Math. J. \textbf{199} (2010), 43--93.

\bibitem {CC2} 
\textsc{Calabri, Alberto, and Ciliberto, Ciro}:
\textit{On Cremona contractibility}, to appear on Rend. Sem. Mat. Univ. Pol. Torino 71 (2013), 11 pages.

\bibitem {Castelnuovo} 
\textsc{Castelnuovo, Guido}:
\textit{Le trasformazioni generatrici del gruppo cremoniano nel piano},
Atti della Reale Accademia delle Scienze di Torino \textbf{36} (1901), 861--874.

\bibitem {Cas2} \textsc{Castelnuovo, Guido}:
 \textit{Ricerche generali sopra i sistemi lineari di curve piane}, Memorie della R. Acc. delle Sc. di Torino (2)  \textbf{42} (1891), 3--43. 

\bibitem {CE} 
\textsc{Castelnuovo, Guido, and Enriques, Federigo}:
\textit{Sulle condizioni di razionalit\`a dei piani doppi},
Rend. Circ. Mat. Palermo \textbf{14} (1900), 290--302.

\bibitem {Cl} \textsc{Clifford, William}: \textit{Analysis of Cremona's transformations} in Math. Papers, Macmillan, London, 1882, 538--542. 

\bibitem{Conforto}
\textsc{Conforto, Fabio}:
\textit{Sui piani doppi razionali},
Rend. Sem. Mat. Univ. Roma (4) \textbf{2} (1938), 156--172.

\bibitem{Coolidge}
\textsc{Coolidge, Julian L.}:
\textit{A treatise on algebraic plane curves},
Dover Publ., New York, 1959.


\bibitem{EC}
\textsc{Enriques, Federigo, and Chisini, Oscar}:
\textit{Lezioni sulla teoria geometrica delle equazioni e delle funzioni
algebriche}, 4 vols., Zanichelli, Bologna, 1915--34.

\bibitem{EConforto}
\textsc{(Enriques, Federigo, and) Conforto, Fabio}:
\textit{Le superficie razionali}, 
Zanichelli, Bologna, 1939--45.

\bibitem{Ferretti} 
\textsc{Ferretti, Giovanni}:
\textit{Sulla riduzione all'ordine minimo dei sistemi lineari di curve piane irriducibili di genere $p$; in particolare per i valori 0, 1, 2 del genere},
Rend. Circ. Mat. Palermo \textbf{16} (1902), 236--279.


\bibitem{Hu}
\textsc{Hudson, Hilda P.}:
\textit{Cremona Transformations in Plane and Space},
Cambridge University Press, 1927.

\bibitem {Ka} \textsc{Kantor, Sigman}:
\textit{Premiers fondements pour une th\'eorie des transformations p\'eriodiques univoques. M\'emoire couronn\'e par l'Acad\`emie des sciences physiqes et math\'ematiques de Naples dans le concours pour 1883}, Acc. Sci. Fis. Mat. e Nat. di Napoli, 1891.  



\bibitem{KT}
\textsc{Kojima, Hideo, and Takahashi, Takeshi}:
\textit{Reducible curves on rational surfaces},
Tokyo J. Math. \textbf{29} (2006), no. 2, 301--317. 


\bibitem{KumarMurthy}
\textsc{Kumar, N. Mohan, and Murthy, M. Pavaman}:
\textit{Curves with negative self-intersection on rational surfaces},
J. Math. Kyoto Univ., \textbf{22} (1982/83), no.\ 4, 767--777.


\bibitem {Iit} \textsc{Iitaka, Shigeru}: \textit{Characterization of two lines on a projective plane}, in Algebraic Geometry, Springer Berlin Heidelberg, (1983), 432--448. 


\bibitem {Iitaka}
\textsc{Iitaka, Shigeru}:
\textit{Classification of reducible plane curves},
Tokyo J. Math. \textbf{11} (1988), no.\ 2, 363--379. 


\bibitem{Jung2}
\textsc{Jung, Giuseppe}:
\textit{Ricerche sui sistemi lineari di genere qualunque
e sulla loro riduzione all'ordine minimo},
Annali di Mat. (2), \textbf{16} (1889), 291--327.

\bibitem{Marletta}
\textsc{Marletta, Giuseppe}:
\textit{Sulla identit\`a cremoniana di due curve piane},
Rend. Circolo Mat. Palermo, \textbf{24} (1907), no.\ 1, 223--242. 

\bibitem{Marletta2}
\textsc{Marletta, Giuseppe}:
\textit{Sui sistemi aggiunti dei varii indici alle curve piane},
Rend. del R. Ist. Lombardo di Scienze e Lettere, (2) \textbf{43} (1910), 781--804. 

\bibitem {Noe}
\textsc{Noether, Max}:
\textit{ \"Uber die auf Ebenen eideutig abbildbaren algebraischen Fl\"achen}, G\"ottingen Nach., (1870), 1--6. 

\bibitem {Noe2}
\textsc{Noether, Max}:
\textit{Zur Theorie der eindeutigen Ebenentransformationen}, Math. Annalen,  \textbf{5} (1872), 635--639. 

\bibitem{Pompilj}
\textsc{Pompilj, Giuseppe}:
\textit{Sulle trasformazioni cremoniane del piano che posseggono una curva di punti uniti},
Rend. del Sem. Mat. dell'Univ. di Roma (4) \textbf{2} (1937), 3--43.

\bibitem {Ros}
\textsc{Rosanes, Jakob}:
\textit{Ueber rationale Substitutionen, welche eine rationale Umkehrung zulassen}, Crelle J., \textbf{73} (1870), 97--110.

\bibitem {Seg}
\textsc{Segre, Corrado}:
\textit{Un'osservazione relativa alla riducibilit\`a delle trasformazioni Cremoniane e dei sistemi lineari di curve piane per mezzo di trasformazioni quadratiche}, Atti Accad. Sci. Torino, \textbf{36} (1901), 645--651.




\end{thebibliography}
\end{document}